\documentclass{article}
\usepackage{graphicx,amsmath,amsfonts,amssymb,amsthm,xfrac,mathtools,dsfont} 
\usepackage[margin=1in]{geometry}

\title{Gelfand-type problem for turbulent jets: \\ sharp $L^\infty$ bound on extremal solutions}
\author{Alex Czemerinski, Alexander Mikheyenko, Noah Tannas, Philip Yao}
\date{\today}
\newtheorem{thm}{Theorem}[section]
\newtheorem{lem}[thm]{Lemma}
\newtheorem{prop}[thm]{Proposition}
\begin{document}

\maketitle
\begin{abstract}
    We consider a model of thermal explosion in reactive turbulent jets introduced in [GHH], which falls into the general class of Gelfand-type problems. Our main result closes the gap in [GMN], where the question of finding an explicit sharp bound for the $L^{\infty}$ norm of the extremal solution in the asymptotic limit of strong flows was left partially unresolved. Namely, we obtain the matching upper and lower bounds under no additional assumptions on the nonlinear reaction rate.
\end{abstract}
\section{Introduction}
We consider the following boundary value problem:
\begin{equation}
        \begin{cases}-\Delta u - \alpha r \varphi (r)\frac{\partial}{\partial r} u=\lambda \psi (r)f(u) &\text{in }B,
        \\ u>0 & \text{in }B,
        \\ u=0 & \text{on }\partial B,
        \end{cases}
    \end{equation}
    where $B\subset \mathbb R^2$ is the unit disk centered at the origin; $\lambda,\alpha>0$ are parameters; $f>0$ is an increasing, convex, and $C^1$ function satisfying
    \[K\coloneqq\int_0^\infty \frac{\mathrm{d}s}{f(s)}<\infty;\]
    and $\varphi,\psi\geq 0$ are decreasing Lipschitz functions on $[0,1]$ such that $\varphi(0)=\psi(0) = 1$, $\varphi>0$ on $[0,1)$, and 
    \[\int_0^1 M(s)\mathrm{d}s < \infty,\]
    where
    \[M(s) \coloneqq \max_{r\in[0,s]}\frac{\psi(r)}{\varphi(r)}.\]
    This problem was derived in [GHH] as a model for the thermal explosion in a reactive turbulent jet. In the context of this model, $u$ is an appropriately normalized temperature, $f$ describes the reaction rate, $\varphi$ and $\psi$ are the flow velocity and reactive component profiles, respectively, $\alpha$ is the injection velocity, and $\lambda$ is the Frank-Kamenetskii parameter characterizing the intensity of the reaction. \\ \newline 
    In [GMN, Proposition 1.1] it was shown that there exists $\lambda^* = \lambda^*(\alpha)\in(0,\infty)$ such that (1) admits a unique minimal classical solution $u^\#_{\lambda,\alpha}$ for $\lambda <\lambda^*$, an extremal solution $u^*_\alpha$ defined as
    \[u_\alpha^*(x) = \lim_{\lambda\nearrow \lambda^*}u^\#_{\lambda,\alpha}(x)\]
    which is also a classical solution for $\lambda = \lambda^*$, and no classical solutions when $\lambda>\lambda^*$.
    Additionally, it was shown that such solutions are radially symmetric and satisfy the semi-stability condition
\begin{equation}
    \int_B |\nabla\eta(\textbf{x})|^2\mu(r)\mathrm{d}\textbf{x} \geq \lambda \int_B\eta^2(\mathbf x)\psi(r)\mu(r)f'(u^\#_{\lambda,\alpha}(r))\mathrm{d}\textbf{x},\hspace{.5cm}\forall \eta\in H_0^1(B),
\end{equation}
where
\begin{equation*}
    \mu(r) \coloneqq \exp\left(\alpha\int_0^r s\varphi(s)\mathrm{d}s\right).
\end{equation*}
In [GMN, Theorem 1.1], it was shown that 
\begin{equation}
    \lim_{\alpha\to\infty}\lambda^*(\alpha)\left(\frac{2K\alpha}{\log\alpha}\right)^{-1}=1.
\end{equation}
In [GMN, Theorem 1.2] it was shown $u_\alpha^*(x)\to0$ for $x\neq0$ and $u_\alpha^*(0)\to\infty$ as $\alpha\to\infty$. An upper bound on $u_\alpha^*(0)$ was established only under additional assumptions on the non-linearity. We show the desired upper bound holds without such additional assumptions and prove a matching lower bound. 
\\ \\
In what follows, fix arbitrary $\varphi, \psi,$ and $f$ satisfying the above conditions.
\begin{thm}
   There exist $\alpha_0,c,C\in (0,\infty)$ such that for all $\alpha\geq\alpha_0$, 
   \[u_\alpha^*(0)\leq cA,\]
   where $A$ solves
   \[f'(A) = C\log \alpha. 
   \]
\end{thm}
\begin{thm}
    There exist $\tilde{\alpha}_0,\tilde{c},\widetilde{C}\in(0,\infty)$ such that for all $\alpha\geq\tilde{\alpha}_0$, we have
    \[u_\alpha^*(0) \geq \tilde{c}\tilde{A},\]
    where $\tilde{A}$ solves
    \[f'(\tilde{A}) = \widetilde{C}\log \alpha.\]
\end{thm}
\noindent In what follows, we refer to $u_\alpha^*(r)$ as just $u(r)$. 
\section{Preliminaries}
Since the extremal solution is radially symmetric, its Laplacian is radial, so the equation in (1) for $\lambda=\lambda^*$ becomes

    \[
        -\frac 1r\frac{\partial}{\partial r}\left(r\frac{\partial u}{\partial r}\right) - \alpha r\varphi(r)\frac{\partial u}{\partial r} = \lambda^*\psi(r)f(u).
    \]
    Now we multiply this by $r\mu(r)$:
    \[
        -\mu(r)\frac{\partial}{\partial r}\left(r\frac{\partial u}{\partial r}\right) - \alpha r\varphi(r)\mu(r)\cdot r\frac{\partial u}{\partial r} = \lambda^*\mu(r)\psi(r)f(u)r.
    \]
    Observe that the left side is equal to
    \[
        -\frac{\partial}{\partial r}\left(\mu(r)\cdot r\frac{\partial u}{\partial r}\right).
    \]
    Integrating from 0 to $r$, then dividing by $r\mu(r)$, we obtain the following problem for $r\in(0,1]$:
    \begin{equation}\label{c}
        \begin{cases}\frac{\partial u}{\partial r} = -\frac {\lambda^*}{r\mu(r)}\int_0^r\mu(s)\psi(s)f(u(s))s\,\mathrm{d}s & r\in (0,1), \\
        u(1) = 0.\end{cases}
    \end{equation}
    To simplify matters, we make the substitution $z = \alpha r^2$ and denote for any function $g$:
    \[\bar g(z) = g\left(\sqrt{\frac z\alpha}\right).\]
    Changing variables appropriately and letting 
    \begin{equation}\label{epsilon}
        F \coloneqq \frac{\lambda^*}{4\alpha} f,
    \end{equation} \eqref{c} becomes
    \begin{equation}\label{d}
        \bar u'(z) = -\frac{1}{z\bar\mu(z)}\int_0^z\bar\mu(s)\bar\psi(s)F(\bar u(s))\mathrm{d}s
    \end{equation}
    with $\bar u(\alpha) = 0$. \\ \\ For radially symmetric $\eta$, (2) becomes
    \begin{equation}
        \int_0^\alpha z\bar\eta'(z)^2\bar\mu(z)\mathrm{d}z\geq \int_0^\alpha\bar\eta(z)^2\bar\psi(z)\bar\mu(z)F'(\bar u(z))\mathrm{d}z.
    \end{equation}
\section{Proof of Theorem 1.1} 
To prove the upper bound, we use (7) along with the behavior of the equation near zero to bound $F(\bar u(0))$ by a constant multiple of $\tilde{u}(0)$, from which the result follows. \\ \\
Note that
\begin{equation}
    \bar\mu(z) = \exp\left(\alpha\int_0^{\sqrt{\frac z{\alpha}}}\varphi(s)s\mathrm{d}s\right) = \exp\left(\frac 12\int_0^z \bar\varphi(s)\mathrm{d}s\right).
\end{equation}
Since $\psi$ is continuous and $\psi(0) = 1$, there exists some $\alpha_1$ big enough such that
\[\psi_0\coloneqq \psi\left(\sqrt{\frac1{\alpha_1}}\right) > 0.\]
Then, since $\psi$ is decreasing, for all $\alpha \geq \alpha_1$ and $z\in[0,1]$ we have
\[ \bar\psi(z)\geq\bar\psi(1)=\psi\left(\sqrt{\frac1\alpha}\right)\geq\psi\Bigg(\sqrt{\frac1{\alpha_1}}\Bigg),
\]
so
\begin{equation}
    0 < \psi_0 \leq \bar\psi(z)\leq 1.
\end{equation}
A few more quick observations about our setup:
\begin{itemize}
    \item From (6), we have that $\bar u' \leq 0$, i.e., $\bar u$ is decreasing. Since $F$ is increasing, this means that $F(\bar u(z))$ is decreasing.
    \item Since $0\leq \bar\varphi(s) \leq 1$, we have from (8) that
    \begin{equation}
        1\leq \bar\mu(s)\leq e^{1/2},
    \end{equation}
    for all $s\in[0,1]$.
\end{itemize}
We will use the semi-stability condition to obtain the following bound:
\begin{lem}
    For all $z\in(0,1]$, we have $\int_0^z F'(\bar u(s))\mathrm{d}s\leq \frac {e^{1/2}}{\psi_0\log 1/z}.$
\end{lem}
\begin{proof}
    For each $z\in (0,1)$, define $\bar\eta_z:[0,\alpha]\to[0,1]$ as follows:
    \[\bar\eta_z(s) = \begin{cases}
        1 & 0\leq s\leq z, \\
        \frac{\log 1/s}{\log 1/z} & z<s\leq 1, \\
        0 & s > 1.
    \end{cases}\]
    One can verify that $\eta_z(r)=\bar\eta_z(\alpha r^2)\in H_0^1(B)$. For this $\bar\eta_z$ we have
    \begin{equation}
        \int_0^\alpha \bar\eta_z(s)^2\bar\psi(s)\bar\mu(s)F'(\bar u(s))\mathrm{d}s\geq \int_0^z\bar\psi(s)\bar\mu(s)F'(\bar u(s))\mathrm{d}s\geq \psi_0\int_0^z F'(\bar u(s))\mathrm{d}s
    \end{equation}
    and
    \begin{equation}
        \int_0^\alpha s\bar\eta_z'(s)^2\bar\mu(s)\mathrm{d}s\leq e^{1/2}\int_z^1 s\bar\eta_z'(s)^2\mathrm{d}s.
    \end{equation}
    For $s\in[z,1]$, we have:
    \[\bar\eta_z'(s) = \frac 1{\log \frac 1z}\cdot \frac{-1}s,\]
    so
    \[e^{1/2}\int_z^1 s\bar\eta_z'(s)^2\mathrm{d}s = \frac {e^{1/2}}{\left(\log\frac 1z\right)^2}\int_z^1 \frac{\mathrm{d}s}s = \frac {e^{1/2}}{\left(\log\frac 1z\right)^2}[\log 1 - \log z] = \frac{e^{1/2}}{\log\frac 1z}.\]
    Combining inequalities (7), (11), and (12) yields
    \[\psi_0\int_0^z F'(\bar u(s))\mathrm{d}s \leq \frac{e^{1/2}}{\log \frac 1z},\]
    as desired.
\end{proof}
\begin{lem}
    For all $z\in(0,1]$, we have $-\bar u'(z)\geq \psi_0e^{-1/2}F(\bar u(z))$.
\end{lem}
\begin{proof}
    Combining (9) and (10) with the decreasing property of $F(\bar u(z))$, we get for all $s\in[0,z]$ that
    \[\bar\psi(s)\bar\mu(s)F(\bar u(s)) \geq \psi_0F(\bar u(z)).\]
    Noting also that
    \[\frac 1{\bar\mu(z)}\geq e^{-1/2},\]
    we see that
    \[-\bar u'(z) \geq \frac {e^{-1/2}}z\int_0^z \psi_0F(\bar u(z))\mathrm{d}s = \psi_0e^{-1/2}F(\bar u(z)).\]
\end{proof}
\begin{lem}
    For all $z\in(0,1]$, we have $-\bar u'(z)\leq e^{1/2}F(\bar u(0))$.
\end{lem}
\begin{proof}
    Similarly, combining (9) and (10) with the decreasing property of $F(u(z))$, we get for all $s\in[0,z]$ that
    \[\bar\psi(s)\bar\mu(s)F(\bar u(s)) \leq e^{1/2}F(\bar u(0)).\]
    Noting also that
    \[\frac 1{\bar\mu(z)}\leq 1,\]
    we see that
    \[-\bar u'(z)\leq \frac 1z\int_0^z e^{1/2}F(\bar u(0))\textrm{d}s = e^{1/2}F(\bar u(0)).\]
\end{proof}
    \begin{lem}
        For all $z\in(0,1]$, we have $F(\bar u(z))\geq F(\bar u(0))\left[1 - \frac e{\psi_0\log 1/z}\right]$.
    \end{lem}
    \begin{proof}
        Dividing both sides of the inequality in Lemma 3.2 by $e^{1/2}F(\bar u(0))>0$, we get
        \begin{align*}
            \frac{-\bar u'(s)}{e^{1/2}F(\bar u(0))}\leq 1,\hspace{.5cm} s\in[0,z].
        \end{align*}
        Hence
        \begin{equation}
            \frac {-1}{e^{1/2}F(\bar u(0))}\int_0^zF'(\bar u(s))\bar u'(s)\mathrm{d}s \leq \int_0^z F'(\bar u(s))\mathrm{d}s.
        \end{equation}
        Using the Fundamental Theorem of Calculus, we see that
        \[\frac {-1}{e^{1/2}F(\bar u(0))}\int_0^zF'(\bar u(s))\bar u'(s)\mathrm{d}s =  \frac{-1}{e^{1/2}F(\bar u(0))}[F(\bar u(z))-F(\bar u(0))].\]
        From (13), combined with Lemma 3.1, this means
        \begin{align*}
            \frac{-1}{e^{1/2}F(\bar u(0))}[F(\bar u(z))-F(\bar u(0))] &\leq \frac{e^{1/2}}{\psi_0\log\frac 1z} \\
            \implies
            F(\bar u(z)) - F(\bar u(0)) &\geq \frac{-eF(\bar u(0))}{\psi_0\log\frac 1z} \\
            \implies F(\bar u(z)) &\geq F(\bar u(0))\left[1 - \frac e{\psi_0\log \frac 1z}\right].
        \end{align*}
    \end{proof}
    \begin{proof}[Proof of Theorem 1]
    Integrating the inequality in Lemma 3.3 from 0 to $z$ gives
    \[\bar u(z)\leq \bar u(0) - \frac{\psi_0}{e^{1/2}}\int_0^{z} F(\bar u(s))\mathrm{d}s.\]
    From Lemma 3.4 we have
    \begin{align*}
        -\frac{\psi_0}{e^{1/2}}\int_0^{z} F(\bar u(s))\mathrm{d}s &\leq -\frac{\psi_0F(\bar u(0))}{e^{1/2}}\int_0^{z}\left(1 - \frac e{\psi_0\log\frac 1s}\right)\mathrm{d}s.
    \end{align*}
    Note that
    \[\lim_{s\to 0^+}\frac{e}{\psi_0\log\frac 1s} = 0,\]
    meaning $z$ can be chosen sufficiently small so that for all $s\in(0,z)$, we have
    \[\frac{e}{\psi_0\log\frac 1s} < 1.\]
    For this $z$, we obtain
    \[\bar u(z) \leq \bar u(0) - I\cdot F(\bar u(0)),\]
    where
    \[I \coloneqq \frac{\psi_0}{e^{1/2}}\int_0^{z}\left(1-\frac e{\psi_0\log\frac 1s}\right)\mathrm{d}s > 0.\]
    Since $\bar u(z)\geq 0$, this means
    \[\bar u(0) - I\cdot F(\bar u(0))\geq 0\implies F(\bar u(0)) \leq \frac{\bar u(0)}I.\]
    By the Fundamental Theorem of Calculus and the convexity of $F$, we have:
    \[F(\bar u(0)) = F\left(\frac{\bar u(0)}2\right) + \int_{\frac{\bar u(0)}2}^{\bar u(0)}F'(s)\mathrm{d}s\geq \int_{\frac{\bar u(0)}2}^{\bar u(0)}F'\left(\frac{\bar u(0)}2\right)\mathrm{d}s = \frac{\bar u(0)}2F'\left(\frac{\bar u(0)}2\right),\]
    so
    \[F'\left(\frac{\bar u(0)}2\right) \leq \frac 2I.\]
    Substituting back \eqref{epsilon} and observing that $\bar u(0) = u(0)$ yields
    \[f'\left(\frac{u(0)}2\right)\leq \frac 8I\cdot \frac \alpha{\lambda^*}.\]
    By (4), one can find $\alpha_2$ such that for all $\alpha\geq\alpha_2$,
    \[\frac{\lambda^*}{\alpha}\geq \frac 12\cdot \frac{2K}{\log\alpha} = \frac K{\log\alpha},\]
    from which we obtain, for all $\alpha\geq\alpha_0= \max\{\alpha_1, \alpha_2\}$,
    \[f'\left(\frac{u(0)}2\right)\leq\frac8{IK}\log\alpha.\]
    So, if $A$ is such that
    \[f'(A)=\frac8{IK}\log\alpha,\]
    then since $f'$ is increasing (since $f$ is convex),
    \[u_\alpha^*(0)\leq2A,\]
    as desired.
    \end{proof}
    \section{Proof of Theorem 1.2}
    The approach for the lower bound involves perturbing a solution to \eqref{d} by some small $v$ such that $v(0) > 0$ and $w_1\coloneqq \bar{u}+v$ still satisfies the equation in \eqref{d}. Assuming Theorem 1.2 is false, we can find an $\alpha$ and corresponding $v$ satisfying $v(z) > 0$ for all $z\in[0,\alpha]$. Then $w_1$ can be used to show the existence of a solution corresponding to a value of $\lambda$ strictly larger than $\lambda^*$, contradicting the extremality of $\lambda^*$. \\ \\ 
    Replacing $\Bar{u}$ with $\bar{u}+v$ in \eqref{d} and rearranging, we obtain the following relation for $v$:
    \begin{equation*}
        v'(z) = -\frac 1{z\bar \mu(z)}\int_0^z \bar\psi(s)\bar\mu(s)[F((\bar u+v)(s))-F(\bar u(s))]\mathrm{d}s.
    \end{equation*}
    To avoid domain issues, we extend the domain of $f$ and $F$ to all of $\mathbb R$ by defining
    \begin{align*}
        \tilde f(t) &= \begin{cases}f(t) & t\geq 0, \\ f(0) & t < 0,\end{cases}
    \end{align*}
    and defining $\widetilde F$ similarly. First, we will need the following result concerning positivity of $v$:
    \begin{prop}
        If Theorem 1.2 is false, then for sufficiently small $v_0>0$  there exists $\alpha$ such that the problem
        \begin{equation}\label{e}
        \begin{cases}v'(z) = -\frac 1{z\bar \mu(z)}\int_0^z \bar\psi(s)\bar\mu(s)[\widetilde F((\bar u+v)(s))-\widetilde F(\bar u(s))]\mathrm{d}s & z\in(0,\alpha), \\ v(0) = v_0, \end{cases}
        \end{equation}
        where $\bar{u}$ solves \eqref{d}, has a solution $v$ with $v(z) > 0$ for all $z\in[0,\alpha]$.
    \end{prop}
    To prove this, we use the following lemma, whose purpose will become apparent.
    \begin{lem}
    Define
    \[\kappa(s) \coloneqq \bar\psi(s)\bar\mu(s)\int_s^{\alpha}\frac{\mathrm{d}z}{\bar\mu(z)}.\]
    Then there exists $k$ independent of $\alpha$ such that for all sufficiently large $\alpha$,
    \[\int_0^t\kappa(s)\mathrm{d}s \leq kt\]
    for all $t\in[0,\alpha]$.
\end{lem}
\begin{proof}
    Define
    \begin{align*}
        R\coloneqq \int_0^1M(s)\mathrm{d}s.
    \end{align*}
    Then, first observe that
    \begin{align*}
        \int_0^\alpha \kappa(s)\mathrm{d}s &= \int_0^\alpha \bar\psi(s)\bar\mu(s)\int_s^\alpha\frac{1}{\bar\mu(z)}\mathrm{d}z\mathrm{d}s \\
        &\leq \int_{0}^\alpha\int_{0}^z\bar\varphi(s)\bar\mu(s)\mathrm{d}s\frac{\overline M(z)}{\bar\mu(z)}\mathrm{d}z \\
        &=2\int_0^\alpha\frac{\bar\mu(z)-\bar\mu(0)}{\bar\mu(z)}\overline M(z)\mathrm{d}z\\
        &\leq 2\int_0^\alpha \overline M(z)\mathrm{d}z \\
        &= 2\int_0^\alpha M\biggl(\sqrt{\frac z\alpha}\biggr)\mathrm{d}z.
    \end{align*}
    Making the substitution $u = \sqrt{z/\alpha}$, the above quantity becomes
    \begin{align*}
        2\alpha\int_0^1uM(u)\textrm du\leq 2\alpha\int_0^1 M(u)\textrm{d}u = 2R\alpha,
    \end{align*}
    so
    \[\int_0^\alpha\kappa(s)\mathrm{d}s \leq 2R\alpha.\]
    Now we pick $0<x<\frac12$ such that $\varphi(\sqrt{2x})=\bar\varphi(2x\alpha)\geq \frac 12$. Then for all $s\leq x\alpha$,
    \begin{align*}
        \kappa(s) &\leq \int_s^{\alpha}\frac{\bar\mu(s)}{\bar\mu(z)}\mathrm{d}z =\left(\int_s^{2x\alpha}+\int_{2x\alpha}^\alpha\right) \frac{\mathrm{d}z}{\exp\left(\frac 12\int_s^z \bar\varphi(u)du\right)}.
    \end{align*}
    Since $\bar\varphi(u)\geq \frac 12$ when $u<2x\alpha$, we can obtain a constant upper bound on the first integral:
    \begin{align*}
        \int_s^{2x\alpha}\frac{\mathrm{d}z}{\exp\left(\frac 12\int_s^z \bar\varphi(u)du\right)} \leq \int_s^{2x\alpha}\frac{\mathrm{d}z}{e^{\frac 14(z-s)}} \leq \int_0^{\infty}\frac{\mathrm{d}z}{e^{z/4}} = 4.
    \end{align*}
    We can obtain a exponentially small bound on the second integral:
    \[\int_{2x\alpha}^\alpha\frac{\mathrm{d}z}{\exp(\frac 12\int_s^z\bar\varphi(u)du)} \leq \frac 1{\exp(\frac 12\int_{x\alpha}^{2x\alpha}\bar\varphi(u)du)}\int_{2x\alpha}^\alpha\frac{\mathrm{d}z}{\exp(\frac 12\int_{2x\alpha}^z\bar\varphi(u)du)}\leq \frac 1{e^{x\alpha/4}}\int_{2x\alpha}^\alpha \mathrm{d}z = \frac{\alpha(1-2x)}{e^{x\alpha/4}}.\]
    Let $\alpha_3$ be such that \[\frac{\alpha(1-2x)}{e^{x\alpha/4}}\leq 1\] for all $\alpha \geq \alpha_3$. Then for such $\alpha$ we obtain
    \[\kappa(s)\leq 5\]
    for all $s\leq x\alpha$, meaning for all $t\leq x\alpha$,
    \[\int_0^t\kappa(s)\mathrm{d}s\leq 5t.\]
    For all $t\geq x\alpha$,
    \[\int_0^t\kappa(s) \mathrm{d}s \leq \int_0^\alpha \kappa(s)\mathrm{d}s \leq 2R\alpha = \frac {2R}x\cdot x\alpha\leq \frac {2R}xt.\]
    Setting $k = \max\left\{5, \frac {2R}x\right\}$ finishes the lemma.
\end{proof}
\begin{proof}[Proof of Proposition 4.1]
    Assuming Theorem 1.2 is false, in particular for $\tilde c=\frac 12$, means that for any $\widetilde C > 0$, there exist arbitrarily large $\alpha$ satisfying
    \[f'(2u(0)) < \widetilde C\log \alpha.\]
    From (4) we see that for sufficiently large $\alpha$,
    \[\lambda^*  < 2\cdot \frac{2K\alpha}{\log\alpha}.\]
    Multiplying this with the above yields
    \[F'(2u(0)) < K\widetilde C.\]
    Setting $\delta = K\widetilde C$ yields the following statement: for all $\delta >0$, there exist arbitrarily large $\alpha$ such that
    \begin{equation}
        F'\bigl(2u(0)\bigr) < \delta.
    \end{equation}
    We will use this form of the negation.\\ \\
    Existence of a $v$ solving (14) is given by Proposition 5.2 in the Appendix. Negation gives, for each $\delta > 0$, an $\alpha$ satisfying both Lemma 4.2 and (15). For small enough $\delta$ and corresponding $\alpha$, we show this $v$ must be positive on $[0,\alpha]$. \\ \\
    Assume for the sake of contradiction that there exists $z_0\in(0,\alpha]$ such that $v(z_0) = 0$. Since $v(0)$ is positive, $z_0$ can be chosen such that $v(z) > 0$ for all $z < z_0$. Then for any $s\in [0,z_0]$, by the Mean Value Theorem there exists $\xi(s)\in[\bar u(s), \bar u(s) + v(s)]$ such that
    \begin{equation}\label{bb}
        F((\bar u + v)(s)) - F(\bar u(s)) = v(s)F'(\xi(s)).
    \end{equation}
    From $F'$ increasing, we have
    \[v(s)F'(\xi(s)) \leq v(s)F'(\bar u(s)+v(s)) \leq v(s)F'(\bar u(0) + v(0)).\]
    If
    \[v(0) < u(0),\]
    then combining (15), \eqref{bb}, and these inequalities we obtain
    \[F((\bar u + v)(s)) - F(\bar u(s)) \leq \delta v(s),\]
    so
    \[-v'(z) \leq \frac{\delta}{z\bar\mu(z)}\int_0^z \bar\psi(s)\bar\mu(s)v(s)\mathrm{d}s.\]
    Using integration by parts, we get
    \begin{align*}
        \int_0^{z_0} v(z)\mathrm{d}z &= -\int_0^{z_0}zv'(z)\mathrm{d}z \\
        &\leq\delta\int_0^{z_0}\int_0^z \frac{\bar\psi(s)\bar\mu(s)v(s)}{\bar \mu(z)}\mathrm{d}s\mathrm{d}z \\
        &= \delta\int_{0}^{z_0}\int_{s}^{z_0}\frac{\bar\psi(s)\bar\mu(s)v(s)}{\bar \mu(z)}\mathrm{d}z\mathrm{d}s \\
        &\leq \delta\int_0^{z_0}\kappa(s)v(s)\mathrm{d}s.
    \end{align*}
    Observe
    \[v(s) = -\int_s^{z_0}v'(t)\mathrm{d}t  = \int_0^{z_0}\mathds 1_{[0,t]}(s)\cdot(-v'(t))\mathrm{d}t .\]
    Thus,
    \begin{align*}
        \delta\int_0^{z_0}\kappa(s)v(s)\mathrm{d}s &= \delta\int_0^{z_0}\int_0^{z_0}\mathds 1_{[0,t]}(s)\kappa(s)(-v'(t))\mathrm{d}t \mathrm{d}s.
    \end{align*}
    Integrating with respect to $s$ first and applying Lemma 4.2, we see that for $t\in[0,z_0]$,
    \[\int_0^{z_0}\mathds 1_{[0,t]}(s)\kappa(s)\mathrm{d}s =  \int_0^t \kappa(s)\mathrm{d}s \leq kt,\]
    so
    \begin{align*}
        \int_0^{z_0}v(z)\mathrm{d}z \leq k\delta \int_0^{z_0}t\cdot(-v'(t))\mathrm{d}t  = k\delta\int_0^{z_0}v(t)\mathrm{d}t .
    \end{align*}
    Since the integral is positive, we get a contradiction for $\delta < 1/k$.
\end{proof}
\begin{proof}[Proof of Theorem 1.2]
    Assume Theorem 1.2 is false, and take $\alpha$ and $v$ from Proposition 4.1. Setting $w_1 \coloneqq \bar u + v$, we see that $w_1$ solves the problem
    \[(P)_\rho = \begin{cases}w_\rho'(z) = -\frac{\rho\lambda^*}{4\alpha}\cdot\frac{1}{z\bar\mu(z)}\int_0^z\bar\psi(s)\bar\mu(s)\tilde f(w_\rho(s))\mathrm{d}s  & z\in(0,\alpha),\\ w_\rho(0)=\gamma_0\end{cases}\]\\
    for $\rho = 1$, where $w_1(\alpha) > 0$ and $\gamma_0\coloneqq u(0)+v(0)$. Note here that the problem $(P)_\rho$ has a unique solution for any value $\rho > 0$ (see Proposition 5.1 in Appendix, with $h\equiv 0$ and $g = (\rho\lambda^*/4\alpha)f$). First, we show the existence of $\rho' < \infty$ such that $w_{\rho'}(\alpha) < 0$. We have:
    \begin{align*}
        -w_{\rho'}'(z) &= \frac{\rho'\lambda^*}{4\alpha}\cdot\frac{1}{z\bar\mu(z)}\int_0^z\bar\psi(s)\bar\mu(s)\tilde f(w_\rho(s))\mathrm{d}s \\
        &\geq \frac{\rho'\lambda^*}{4\alpha}\cdot \frac{f(0)}{z\bar\mu(z)}\int_0^z\bar\psi(s)\bar\mu(s)\mathrm{d}s.
    \end{align*}
    Since
    \[\lim_{z\to 0^+} \frac 1{z\bar\mu(z)}\int_0^z\bar\psi(s)\bar\mu(s)\mathrm{d}s = 1,\]
    and for all $z\in(0,\alpha]$,
    \[\frac 1{z\bar\mu(z)}\int_0^z\bar\psi(s)\bar\mu(s)\mathrm{d}s \geq y \]
    for some $0<y<1$, we have
    \[-w_{\rho'}'(z) \geq \left(\frac{\lambda^*f(0)}{4\alpha}\cdot y\right)\rho'.\]
    For $\rho'$ large enough, this gives $w_{\rho'}(\alpha) < 0$. \\ \\
    Now we show that $w_\rho(\alpha)$ is continuous with respect to $\rho$ on $[1,\rho']$. If $\rho_1,\rho_2\in[1,\rho']$, then 
    \begin{align*}
        |(w_{\rho_1} - w_{\rho_2})'(z)| &\leq \frac{\lambda^*}{4\alpha}\cdot\frac{1}{z\bar\mu(z)}\int_0^z \bar\psi(s)\bar\mu(s)\Bigl|\rho_1\tilde f(w_{\rho_1}(s)) - \rho_2\tilde f(w_{\rho_2}(s))\Bigr|\mathrm{d}s.
    \end{align*}
    Examining the last term in the integrand and using that $\tilde f$ is Lipschitz on $(-\infty,\gamma_0]$:
    \begin{align*}
        |\rho_1\tilde f(w_{\rho_1}(s)) - \rho_2\tilde f(w_{\rho_2}(s))| &= \Bigl|\rho_1(\tilde f(w_{\rho_1}(s))-\tilde f(w_{\rho_2}(s))) + (\rho_1-\rho_2)\tilde f(w_{\rho_2}(s))\Bigl| \\
        &\leq L\rho'|(w_{\rho_1} - w_{\rho_2})(s)| + |\rho_1-\rho_2|\cdot \tilde f(\gamma_0),
    \end{align*}
    for some finite $L$. It follows that
    \[|(w_{\rho_1} - w_{\rho_2})'(z)|\leq C_1\max_{s\in[0,z]}|(w_{\rho_1} - w_{\rho_2})(s)| + C_2 |\rho_1 -\rho_2|,\]
    for some finite $C_1, C_2$.
    This implies
    \begin{align*}
        |(w_{\rho_1} - w_{\rho_2})(z)| &= \left|\int_0^z(w_{\rho_1} - w_{\rho_2})'(s)\mathrm{d}s\right| \\
        &\leq \int_0^z |(w_{\rho_1} - w_{\rho_2})'(s)|\mathrm{d}s \\
        &\leq C_1\int_0^z\max_{t\in[0,s]}|(w_{\rho_1} - w_{\rho_2})(t)|\mathrm{d}s + \alpha C_2 |\rho_1-\rho_2|. 
    \end{align*}
    Notice that the above inequality holds if $z$ in the first expression is replaced by any $r\in[0,z]$, meaning
    \[\max_{r\in[0,z]}|(w_{\rho_1} - w_{\rho_2})(r)| \leq C_1\int_0^z\max_{t\in[0,s]}|(w_{\rho_1} - w_{\rho_2})(t)|\mathrm{d}s + \alpha C_2 |\rho_1-\rho_2|.\]
    By Grönwall's inequality, we obtain
    \[\max_{r\in[0,\alpha]}|(w_{\rho_1} - w_{\rho_2})(r)| \leq C_2\alpha e^{C_1\alpha}|\rho_1-\rho_2|,\]
    which implies $w_\rho(\alpha)$ is Lipschitz, and therefore continuous with respect to $\rho$. \\ \\
    Since $w_\rho(\alpha)$ is continuous on $[1,\rho']$, $w_1(\alpha) > 0$, and $w_{\rho'}(\alpha)< 0$, the Intermediate Value Theorem guarantees the existence of $\rho^*\in (1,\rho')$ such that
    \[w_{\rho^*}(\alpha) = 0.\]
   If $w(r)\coloneqq w_{\rho^*}(\alpha r^2)$, then $w$ is a solution to (4) with $\lambda = \rho^*\lambda^* > \lambda^*$, contradicting the definition of $\lambda^*$.
    \\ \\
    This completes the proof of Theorem 1.2.
    \end{proof}
    \section{Appendix}
    In this section, we provide and prove existence and uniqueness results for the problems described in Section 4. These results are deferred here since they are standard in the theory of ordinary differential equations.
    \begin{prop}
        Let $g:\mathbb [0,\infty)\to(0,\infty)$ be an increasing convex $C^1$ function; define $\tilde g:\mathbb R\to(0,\infty)$ by
        \[\tilde g(t) = \begin{cases}g(t) & t\geq 0, \\ g(0) & t < 0;\end{cases}\]let $h:[0,\alpha]\to \mathbb R$ be a continuous function; and fix $v_0 \in\mathbb R$. Then there exists a unique $C^1$ function $v:[0,\alpha]\to\mathbb R$ that solves the initial value problem
        \[\begin{cases}v'(z) = -\frac{1}{z\bar\mu(z)}\int_0^z \bar\psi(s)\bar\mu(s)\tilde g(h(s)+v(s))\mathrm{d}s & z\in(0,\alpha), \\ v(0) = v_0,\end{cases}\]
        where $\bar{\psi}$ and $\bar{\mu}$ are as defined in previous sections.
    \end{prop}
    \begin{proof}
        Let $h_0$ be the maximum value of $h$, which exists by the Extreme Value Theorem.
        We have by convexity of $g$ that
        \[\tilde g'(t) \leq \tilde g'(h_0 + v_0)\]
        for all $t \leq h_0 + v_0$, so $\tilde g$ is Lipschitz on $(-\infty, h_0 + v_0]$ with constant $\ell \coloneqq g'(h_0 + v_0)$. \\ \\
        Integrating the equation, we see that $v$ solves the problem if and only if
        \[v(z) = v_0 - \int_0^z \frac 1{s\bar\mu(s)}\int_0^s \bar\psi(t)\bar\mu(t) \tilde g(h(t)+v(t))\mathrm{d}t ,\]
        or equivalently if $v$ is a fixed point of the operator $T$ defined by
        \[Tv(z) = v_0 - \int_0^z \frac 1{s\bar\mu(s)}\int_0^s \bar\psi(t)\bar\mu(t) \tilde g(h(t)+v(t))\mathrm{d}t .\]
        Consider the following subspace of $C^0[0,\alpha]$:
        \[X\coloneqq \{f\in C^0[0,\alpha]: \max_{z\in[0,\alpha]}f(z) \leq v_0\}.\]
        Recall that $C^0[0,\alpha]$ is a complete metric space under 
        \[d(u,v) =\Vert{u-v}\Vert=\max_{z\in[0,\alpha]}|u(z) - v(z)|\]
        and $X$ is closed, so $X$ is complete. \\ \\
        First, we show that $T$ maps $X$ to itself. Indeed, if $v\in X$, then the function
        \[s\mapsto \frac 1{s\bar\mu(s)}\int_0^s \bar\psi(t)\bar\mu(t)\tilde g(h(t)+v(t))\mathrm{d}t .\]
        is continuous on $(0,\alpha]$, so to ensure that $Tv$ is continuous it suffices to check 
        \[\lim_{s\to 0^+}\frac 1{s\bar\mu(s)}\int_0^s \bar\psi(t)\bar\mu(t)\tilde g(h(t) + v(t))\mathrm{d}t\] exists and is finite. One can verify using L'Hôpital's rule that the above limit is given by $\tilde g(h(0) + v_0)$, so we have $Tv\in C^0[0,\alpha]$. \\ \\
        From the definition of $T$, we see that $Tv(z)$ is a decreasing function of $z$ and that $Tv(0) = v_0$, so
        \[\max_{z\in[0,\alpha]}Tv(z) = v_0.\]
        This gives $Tv\in X$. \\ \\
        For any $u,v\in X$,
        \begin{align*}
            |Tu(z) - Tv(z)| &\leq \int_0^z \frac 1{s\bar\mu(s)}\int_0^s \bar\mu(t) \bigl|\tilde g(h(t) + u(t)) - \tilde g(h(t)+v(t))\bigr|\mathrm{d}t \mathrm{d}s \\
            &\leq \ell\int_0^z \frac 1{s\bar\mu(s)}\int_0^s \bar\mu(t)\max_{r\in[0,t]}\bigl|[h(r) + u(r)] - [h(r) + v(r)]\bigr|\mathrm{d}t \mathrm{d}s \\
            &\leq \ell\int_0^z \max_{r\in[0,s]}|u(r) - v(r)|\mathrm{d}s,
        \end{align*}
        for all $z\in[0,\alpha]$. The above bound remains valid if $z$ in the first expression is replaced with any $r\in[0,z]$, so
        \[\max_{r\in[0,z]}|Tu(r) - Tu(r)|\leq \ell\int_0^z \max_{r\in[0,s]}|u(r) - v(r)|\mathrm{d}s.\]
        We claim that for all positive integers $n$,
        \[\max_{r\in[0,z]}|T^nv(r)-T^nu(r)| \leq \frac{(\ell z)^n}{n!}\Vert u-v\Vert,\]
        for all $z\in[0,\alpha]$, from which we will obtain
        \[\Vert T^nv - T^nu\Vert \leq \frac{(\ell \alpha)^n}{n!}\Vert u-v\Vert.\]
        Indeed, we have by induction that for all $z\in[0,\alpha]$,
        \begin{align*}
            \max_{r\in[0,z]}|T^nu(r) - T^nv(r)| &\leq \ell \int_0^z \max_{r\in[0,s]}|T^{n-1}u(r)-T^{n-1}v(r)|\mathrm{d}s \\
            &\leq \frac{\ell^n}{(n-1)!}\Vert u-v\Vert \int_0^zs^{n-1}\mathrm{d}s \\
            &= \frac{\ell^n}{(n-1)!}\Vert u-v\Vert\cdot \frac{z^n}n \\
            &= \frac{(\ell z)^n}{n!}\Vert u-v\Vert.
        \end{align*}
        Since
        \[\lim_{n\to\infty} \frac{(\ell \alpha)^n}{n!} = 0,\]
        there exists a positive integer $N$ such that $T^N$ is a contraction. By the Banach fixed-point theorem, $T^N$ has a unique fixed point $v^*$, satisfying
        \[T^Nv^* = v^*.\]
        Taking $T$ on both sides gives
        \[T^{N+1}v^* = Tv^*\implies T^N(Tv^*) = Tv^*,\]
        so $Tv^*$ is also a fixed point of $T^N$. Since the fixed point is unique, we get
        \[v^* = Tv^*,\]
        so $v^*$ is a fixed point of $T$, i.e. a solution. Also noting that any fixed point of $T$ is a fixed point of $T^N$, we see that $v^*$ is a unique fixed point of $T$. \\ \\
        Finally, we note that $v^*$ is $C^1$ as it is equal to an integral of a $C^0$ function.
    \end{proof}
    \begin{prop}
        The problem
        \[\begin{cases}v'(z) = -\frac 1{z\bar \mu(z)}\int_0^z \bar\psi(s)\bar\mu(s)[\widetilde F((\bar u+v)(s))-\widetilde F(\bar u(s))]\mathrm{d}s & z\in(0,\alpha), \\ v(0) = v_0, \end{cases}\]
        admits a solution for all $\alpha > 0$.
    \end{prop}
    \begin{proof}
        Let
    \begin{equation}\label{aa}
        v_1(z) \coloneqq  -\int_0^z\frac 1{s\bar\mu(s)}\int_0^s\bar\psi(t)\bar\mu(t)\widetilde F(\bar u(t))\mathrm{d}t \mathrm{d}s,
    \end{equation}
    and let $v_2$ solve the problem
    \[\begin{cases}
        v_2'(z) = -\frac 1{z\bar\mu(z)}\int_0^z \bar\psi(s)\bar\mu(s)\widetilde F((\bar u + v_2 - v_1)(s))\mathrm{d}s & z\in(0,\alpha), \\
        v_2(0) = v_0;
    \end{cases}\]
    the existence of $v_2$ is given by Proposition 5.1 with $h = \bar u - v_1$ and $g=F$. Differentiating \eqref{aa} and subtracting yields
    \[(v_2-v_1)'(z) = -\frac 1{z\bar\mu(z)}\int_0^z\bar\psi(s)\bar\mu(s)\bigl[\widetilde F((\bar u + v_2 - v_1)(s)) - \widetilde F(\bar u(s))\bigr]\mathrm{d}s.\]
    Additionally, 
    \[(v_2 - v_1)(0) = v_0 - 0 = v_0.\]
    Thus,
    \[v \coloneqq  v_2 - v_1\]
    solves the desired problem for any value of $\alpha$.
    \end{proof}
    \section{Acknowledgements}
    This project was carried out during the 2026 REU program at the Department of Mathematical Sciences at Kent State University, supported by the National Science Foundation under Grant No. DMS-2439984. The authors would like to thank Peter Gordon and Fedor Nazarov for their guidance and assistance on the project, and Zachary Chase for assistance with making edits on the paper.
    \section{References}
    [GHH] P. V. Gordon, U. G. Hedge, M. C. Hicks, An elementary model for autoignition of free round turbulent jets, SIAM J. Appl. Math 78 (2) (2018), 705-718
    \\ \\
    \noindent [GMN] P. V. Gordon, V. Moroz, F. Nazarov, Gelfand-type problem for turbulent jets, J. Differential Equations 269 (2020), 5959-5996
\end{document}